\documentclass{article}
\usepackage{amsmath}
\usepackage{amsfonts}
\usepackage{amssymb}
\usepackage{pstricks,pst-node,amssymb,amsmath,graphics,latexsym,tabularx,shapepar}
\newtheorem{theorem}{Theorem}

\newtheorem{corollary}[theorem]{Corollary}

\newtheorem{definition}[theorem]{Definition}
\newtheorem{example}[theorem]{Example}

\newtheorem{lemma}[theorem]{Lemma}

\newtheorem{proposition}[theorem]{Proposition}
\newtheorem{remark}[theorem]{Remark}

\newenvironment{proof}[1][Proof]{\noindent\textbf{#1.} }{\ \rule{0.5em}{0.5em}}
\newcommand{\bpartial}{\mathop{\partial\kern -4pt\raisebox{.8pt}{$|$}}}
\newcommand{\bra}{\mathopen{[\kern-1.6pt[}}
\newcommand{\ket}{\mathclose{]\kern-1.5pt]}}
\newcommand{\bbra}{\mathopen{[\kern-2.2pt[\kern-2.3pt[}}
\newcommand{\bket}{\mathclose{]\kern-2.1pt]\kern-2.3pt]}}

\makeindex

\begin{document}
\title {\large{\bf A note on  integrabiliy of Hamiltonian systems on the co-adjoint Lie groupoids}}
\vspace{3mm}
\author {\small{ \bf  Gh. Haghighatdoost$^1$}\hspace{-1mm}{ \footnote{Corresponding author. }, \small{ \bf R. Ayoubi $^2$ }\hspace{-2mm} }\\
{\small{$^{1}$, $^{2}$\em Department of Mathematics,Azarbaijan Shahid Madani University, Tabriz, Iran}}\\
{\small{ E-mails: $^{1}$\em gorbanali@azaruniv.ac.ir, $^{2}$\em rezvaneh.ayoubi@azaruniv.ac.ir }}\\
{\small{$^{1}$\em Telephone: +9831452392, Fax: +9834327541 }}
}
 \maketitle
\begin{abstract}
As we said in our previous work \cite{H A 2}, the main idea of our research is to introduce a class of Lie groupoids by means of co-adjoint representation of a Lie groupoid on its isotropy Lie algebroid,
 which we called co-adjoint Lie groupoids. In this paper, we will examine the relationship between  
 structural mappings of the Lie algebroid associated to Lie groupoid and co-adjoint Lie algebroid. Also, we try to construct and define  integrabiliy of
Hamiltonian system on the co-adjoint Lie groupoids. In addition, we show that co-adjoint Lie groupoid associated to a symplectic Lie groupoid is a symplectic Lie groupoid.\\


\noindent{\bf AMS}: 18B40, 53D17, 70H08, 37J35. 

\end{abstract}

\maketitle

\section{Introduction}
In mathematics, the adjoint representation (or adjoint action) of a Lie group $\mathfrak{G}$ is a method of presenting  the elements of the group as linear transformations of its Lie algebra, which is considered as a vector space. For every Lie group, this natural representation is obtained by linearizing (i.e. taking the differential of) the action of $\mathfrak{G}$ on itself by conjugation.
Also, the co-adjoint representation of a Lie group $\mathfrak{G}$ is the dual of the adjoint representation. If $\mathfrak{g}$ denotes the Lie algebra of $\mathfrak{G}$, the corresponding action of $\mathfrak{G}$ on $\mathfrak{g}^{\ast}$, the dual space to $\mathfrak {g}$, is called the co-adjoint action \cite{Haghighatdoost}. A geometrical interpretation is as the action by left-translation on the space of right-invariant 1-forms on $\mathfrak{G}$.
Kirillov's work has played a major role in demonstrating the importance of co-adjoint representation. In the Kirillov method of orbits, representations of $\mathfrak{G}$ are constructed geometrically starting from the co-adjoint orbits \cite{Kirillov}.

It is well-known that the orbit of co-adjoint representation of a Lie group is symplectic manifold and has natural Poisson structure which called the Kirillov-Kostant structure. The study of symplectic and Poisson structures on the orbits of co-adjoint action of Lie groups is examined in \cite{Bolsinov} and \cite{Trofimov}.

Integrable Hamiltonian systems play an important role in the study of physical systems because of their interesting properties in terms of mathematics and physics. The number of Hamiltonian systems identified to date is very small and there is no general method for expressing whether a Hamiltonian system is integrable or not.
 It is worth noting that integrable Hamiltonian systems are introduced for Lie algebra $e(3)$ in \cite{Bolsinov}, for Lie algebra $so(4)$ and $so(3,1)$ in \cite{HO, Sokolov}.

The groupoids introduced by H. Brandt in 1926 are mathematical structures that can describe the properties of symmetry more general than what described by groups. Subsequently, groupoids with additional structures (topological and distinctive) were used as essential tools in differential topology and geometry. 
As we know, if a Lie groupoid $G$ is endowed with a symplectic form that is multiplicative, then $G$ is called a symplectic groupoid.
Symplectic groupoids can be used for the construction of noncommutative deformations of the algebra of smooth functions on a manifold, with potential applications to quantization \cite{ Marle}.

As we know, there is unfortunately no natural way to to describe the co-adjoint representation of Lie groupoids.  So, the authors, try to solve this by different methods \cite{Lang}. 
Also, as we show in \cite{H A 2}, we introduce the co-adjoint orbits of Lie groupoids by means of co-adjoint representation of a Lie groupoid on its isotropy Lie algebroid. Then we were able to introduce a class of Lie groupoids called co-adjoint Lie groupoids and examine the structure of its associated Lie algebroid. In \cite{H}, the first author has studied the optimal control problems on the co-adjoint Lie groupoids.
As we know, the topic of Hamiltonian mechanics on Lie algebroids is one of the topics that has been highly regarded by the authors and has many applications. (For further reading, you can refer to \cite{H A 1} and its references.)
In this article, which is a continuation of our previous work \cite{H A 2}, we will show the interesting relationship between the structural functions of the Lie algebroid associated to Lie groupoid and the structural functions of the co-adjoint Lie algebroid. Furthermore, given the specific mode of the co-adjoint Lie groupoids, i.e. the co-adjoint the groupoids in the form $O(\xi) = M \times O(\xi^{\prime})$ which $O(\xi^{\prime}) \subset \mathfrak{g}^{\ast}$ is co-adjoint orbit of Lie groups, we intend to define the integrabiliy of Hamiltonian system on the co-adjoint Lie algebroid based on the methods expressed in \cite{Bolsinov, HO, Sokolov}. In addition, we show that
co-adjoint Lie groupoid associated to a symplectic Lie groupoid is a symplectic Lie groupoid.
\section{ Some basic concepts and definitions}
\subsection{ Review of Lie groupoid and Lie algebroid conceptions}
We will first review some of the basic concepts and definitions of Lie groupoids and Lie algebroids, one can refer to \cite{Mackenzie} to see more concepts and details.

\begin{definition}
\label{groupoid}
A groupoid $G$ over $M$, which we will denote by $G \rightrightarrows M$, consists of two sets $G$ and $M$ together with structural mappings $\alpha, \beta, 1, \iota$ and $m,$  where source mapping $\alpha: G \longrightarrow M,$ target mapping $\beta:G \longrightarrow M,$ unit mapping $1: M \longrightarrow G,$ inverse mapping $\iota : G \longrightarrow G $ and multiplication mapping $ m: G_{2} \longrightarrow G $ which $ G_{2} = \lbrace (g,h) \in G \times G ~\vert ~ \alpha (g) = \beta (h) \rbrace $ is subset of $G \times G.$\\

A Lie groupoid is a groupoid $G \rightrightarrows M$ for which $G$ and $M$ are smooth manifolds, $\alpha, \beta, 1, \iota$ and $m,$ are differentiable mappings and besides, $\alpha,\beta$ are differentiable submersions.
\end{definition}

\begin{definition}
\label{bisection}
Let $ G \rightrightarrows M$ be a Lie groupoid. A bisection of $G$ is a smooth map $ \sigma : M \longrightarrow G $  which is right inverse to $\alpha : G \longrightarrow M, ~ ( \alpha \circ \sigma = id_{M} )$ and $\beta \circ \sigma : M \longrightarrow M $ is a diffeomorphism.

Let $\sigma$ be a bisection of $G.$ The left translation corresponding to $\sigma$ is defined as:
\begin{eqnarray}
L_{\sigma} : G & \longrightarrow & G \nonumber \\
g & \longmapsto & \sigma (\beta g) g \nonumber
\end{eqnarray}
and the right translation corresponding to $\sigma$ is defined by:
\begin{eqnarray}
 R_{\sigma} :  G & \longrightarrow &  G   \nonumber  \\
g & \longmapsto & g \sigma \big( (\beta \circ \sigma )^{-1} (\alpha g ) \big). \nonumber
\end{eqnarray}
\end{definition}

Another very useful concept in the subject of groupoids is the definition of action of groupoids, which we will refer to. Also, there is a similar concept of a right smooth action.
\begin{definition}
\label{groupoid action}
Let $ G \rightrightarrows M$ be a Lie groupoid and $J : N \longrightarrow M$ be a smooth map. A smooth left action of Lie groupoid G on $J : N \longrightarrow M$ is a smooth map $\theta : G_{~ \alpha} \times_{~J} N \longrightarrow N $ which has the following conditions.:
\begin{enumerate}
\item For every $(g,n) \in G_{~ \alpha} \times_{~J} N,~~~~ J (g.n) = \beta (g),$ 
\item For every $n \in N,~~~~ 1_{J(n)} .n = n,$
\item For every $(g , g^{\prime}) \in G_{2} $ and $n \in J^{-1} (\alpha (g^{\prime})),~~~~ g . ( g^{\prime} . n ) = (g g^{\prime})  . n$
\end{enumerate}

(where $ g.n := \theta ( g , n )$ and $ \theta (g) (n):= \theta (g , n )$).\\

\end{definition}

\begin{definition}
A Lie algebroid A over a manifold M is a vector bundle  $\tau : A \longrightarrow M$ with the following items:
\begin{enumerate}
\item A Lie bracket $[\vert ~,~ \vert] $ on the space of smooth sections $ \Gamma( \tau),$
$$[\vert ~,~ \vert]: \Gamma (\tau) \times \Gamma (\tau) \longrightarrow \Gamma (\tau) $$
$$(X , Y ) \longmapsto [\vert X , Y \vert].$$
\item A morphism of vector bundles $\rho : A \longrightarrow TM,$ called the anchor map, where $TM$ is is the tangent bundle of $M,$
\end{enumerate}
such that the anchor and the bracket satisfy the following Leibniz rule:
$$[\vert X , f Y \vert] = f [\vert X , Y \vert] + \rho (X) (f) Y$$
where $X , Y \in \Gamma (\tau)$, $~f \in C^{\infty} (M)$ and $\rho (X) f$ is the derivative of $f$  along the vector field $\rho (X).$

\end{definition}

\begin{definition}
\label{associated Lie algebroid}
Let $G \rightrightarrows M $  be a Lie groupoid. Consider the vector bundle  $\tau : AG \longrightarrow M,$ where 
$$AG := ker T\alpha \vert_{1_{p}}$$
where $\alpha : G \longrightarrow M$ is source mapping of the Lie groupoid $G \rightrightarrows M$,  $T\alpha : TG \longrightarrow TM$ is tangent mapping of $\alpha. $\\

It is well-known that $AG$ has Lie algebroid structure as follows:\\
As we know, there exists a bijection between the space of sections $\Gamma (\tau)$ and the set of left (right) invariant vector fields on $G$ (see \cite{Mackenzie}). 
If $X$ be a section of $\tau : AG \longrightarrow M,$ right  invariant and left invariant vector fields corresponding to $X$ are, respectively, defined  by
\begin{eqnarray}
 \overrightarrow{X}(g) &=& T R_{g} ( X (\beta (g))) \nonumber  \\
 \overleftarrow{X}(g) &=& - T (L_{g} ) T (\iota) ( X (\alpha (g)),   \nonumber
\end{eqnarray}
where  $L_{g}$ and $R_{g}$ are left translation and right translation corresponding to $g\in G.$  
\\

Therefore, Lie algebroid structure $([\vert ~,~ \vert], \rho ) $ on $AG$ can be introduced as follows:
\begin{enumerate}
\item The anchor map: 
$$ \rho: AG  \longrightarrow  TM  $$
$$ \rho(X)(x) =T_{1(p)} \beta (X(p)) $$
where $X \in \Gamma (\tau), p \in M.$\\
\item Lie bracket:     
 $$\Gamma (AG) \times \Gamma (AG) \longrightarrow \Gamma (AG)$$
 $$[\vert \overrightarrow{X , Y} \vert] := [ \overrightarrow{X} , \overrightarrow{Y} ], ~~~( [\vert \overleftarrow{X , Y} \vert] := - [ \overleftarrow{X} , \overleftarrow{Y} ] )$$ \\
where $ X , Y \in \Gamma (\tau)$ and $[~,~]$ is standard Lie bracket of vector fields.
\end{enumerate}
\end{definition}

In definition \ref{groupoid action}, we have presented the definition of the action of the Lie groupoid on smooth mapping, now we want to express the definition of action of a Lie algebroid on a smooth mapping.

\begin{definition}
Let $ ( A, M, \pi, \rho , [\vert ~,~ \vert ] )$ be a Lie algebroid and $J : N \longrightarrow M$ be a smooth map.
An action of a Lie algebroid $A$ on map $J : N \longrightarrow M$ is a map
$ \theta : \Gamma (A) \longrightarrow \mathfrak{X} (N) $ which satisfies in the following properties:
 
\begin{enumerate}
\item $ \theta ( X + Y ) = \theta (X) + \theta (Y)$
\item $ \theta ( f X ) = J^{\ast} f \theta (X) $
\item $ \theta  ( [\vert X , Y \vert] ) = [ \theta (X) , \theta (Y) ]  $
\item $T J ( \theta (X) ) = \rho (X)$
\end{enumerate}
 for all $ f \in C^{\infty} (M)$ and $ X, Y \in \Gamma^{\infty} (A).$ 
In addition, $J^{\ast} : C^{\infty} (M) \longrightarrow C^{\infty} (N) $ such that $J^{\ast} f = f \circ J \in C^{\infty} (N)$ is pullback of $f$ by $J$.
\end{definition}

\begin{remark}
\label{action remark}
As mentioned in \cite{Bos},  every action  $ \theta$ of a Lie groupoid $G$ on $ J : N \longrightarrow M $ induce an action $\theta^{\prime}$ of a Lie algebroid $ A(G) $ on $ J : N \longrightarrow M $ as follows:
$$ \theta ^{\prime} (X) (n) := \dfrac{d}{dt} \vert _{t=0} ~Exp ( t X )_{J(n)} . n.$$
\end{remark}

\begin{definition}
The Lie groupoid $G \rightrightarrows M $  is regular Lie groupoid if the anchor
\begin{eqnarray}
( \beta , \alpha ) : G &\longrightarrow & M  \nonumber \\
g &\longmapsto & \big( (\beta (g) , \alpha(g) \big)  \nonumber
\end{eqnarray}
  is a mapping of constant rank.

\end{definition} 

\begin{remark}
Similar to what was said in the previous work \cite{H A 2}, throughout the article, we assume that $G$ be a regular Lie groupoid over $M.$
\end{remark}

\subsection{ The adjoint and co-adjoint actions}
In this section, as described in \cite{Bos}, we will briefly introduce the concepts of adjoint and co-adjoint actions of a Lie groupoid.

Let $G \rightrightarrows M $ be a Lie groupoid. For an arbitrary element $p \in M,$ the isotropy group of $ G \rightrightarrows M $ is defined by $ I_{p} = \alpha ^{-1} (p) \cap \beta ^{-1} (p).$  As we mentioned in \cite{H A 2} $I_{p}$ is a Lie group. Moreover,  $I_{G} = ( \cup I_{p} )_{p \in M}$ is a groupoid over $M$ but it is not a smooth manifold in general, see example A.10 in \cite{Schmeding}.

\begin{lemma}
The associated isotropy groupoid of a regular Lie groupoid  $G \rightrightarrows M $ is a Lie groupoid.
\end{lemma}
\begin{proof}
see \cite{Schmeding}
\end{proof}\\

 For Lie groupoid $G \rightrightarrows M$ we denote the associated isotropy Lie groupoid by $I_{G}$ and the Lie algebroid associated to isotropy Lie groupoid by $A I_{G}$ and call it isotropy Lie algebroid.

Let $ G \rightrightarrows M $ be a Lie groupoid and $I_{G}$ be its associated isotropy Lie groupoid. $G$ acts smoothly from the left on $ J: I_{G} \longrightarrow M $ by conjugation, $$ C(g) (g^{\prime} ):= g g^{\prime} g^{-1}.$$

 On the other hand, the conjugation action induces an action of a Lie groupoid $G$ on $ A I_{G} \longrightarrow M $. We call this action \textbf{adjoint action} of $G$ on $A I_{G}$ which can be defined as 
 $$Ad: G \times A I_{G} \longrightarrow AI_{G} $$
$$Ad_{g} X := \dfrac{d}{dt} \Big\vert_{t=0}~ C(g) Exp (tX)$$
where  $p \in M$,\; $g \in G_{p} = \alpha^{-1} (p) $ (the $\alpha$-fibers over $p$)\;  and\;  $X \in (A I_{G})_p.$\\

According to remark \ref{action remark}, the action $Ad$ induces an adjoint action of $AG $ on $ A I_{G}  \longrightarrow M $ as 
$$ad: AG \times AI_G \to AI_G$$
$$ ad_XY=ad (X)  (Y) := \dfrac{d}{dt} \Big\vert_{t=0}~ Ad (Exp (tX) ) Y $$
where $ X \in (AG)_{p} ,~ Y \in (A I_{G})_{p}$ and $ p \in M.$ \\
It is clear that $ X \in \Gamma(AG) $ and $Y \in  \Gamma(A I_{G} )$
$$ ad_X (Y) = [\vert X , Y \vert].$$ 

Another action of $G$  on dual bundle $A^{\ast} I_{G} $ which is called \textbf{co-adjoint action} of $G,$ for $ g \in G_{p},~ \xi \in (A^{\ast} I_{G})_p,$ is defined as 
$$ Ad^{\ast} : G \times A^{\ast} I_{G} \longrightarrow A^{\ast} I_{G}  $$
$$ Ad_{g} ^{\ast} \xi (X) := \xi ( Ad_{g^{-1}} X ).  $$
In other words $$ ~~~~~~~~\langle Ad_{g} ^{\ast} \xi , X \rangle = \langle \xi , Ad_{g^{-1}} X\rangle. $$
Also, according to remark \ref{action remark}, the action $Ad^{\ast}$ induces so-called co-adjoint action of a Lie algebroid $AG$ on $A^{\ast} I_{G} $ which for $ \xi \in (A^{\ast} I_{G})_p $ is defined by
$$ ad^{\ast} : AG \times A^{\ast} I_{G}  \longrightarrow A^{\ast} I_{G} $$
$$ ad_{X} ^{\ast} \xi (Y) := \xi ( ad_{-X} (Y) ) = \xi ( [\vert Y,X \vert] ) $$
or $$ ~~~~~~~~\langle ad_{X} ^{\ast} \xi , Y  \rangle = \langle \xi , ad (- X) Y \rangle. $$
\section{Co-adjoint Lie groupoid and co-adjoint Lie algebroid}
In this section, we briefly review concepts about the Lie groupoid corresponding to the orbits of co-adjoint action of Lie groupoid and its associated Lie algebroid. One can refer to \cite{H A 2} to see and study more details.\\

\begin{definition}
\label{co orbit}
Let $G \rightrightarrows M $ be a  Lie groupoid with source, target, multiplication, unit and inverse mapping $\alpha, ~ \beta, ~ m,~1,~ \iota,$ respectively. Suppose that $\xi$ is an element of $(A^{\ast} I_{G})_{p},$ we define the orbit of co-adjoint action of a Lie groupoid $G$ as follows:
$$O(\xi) = \lbrace Ad_{g} ^{\ast} \xi ~\vert ~ g \in G \rbrace$$
\end{definition}
It turns out that for all $\xi$ that the stabilizer $G_\xi=\{g : Ad^\ast_g \xi=\xi  \}$ is a normal Lie subgroupoid of $G$, the co-adjoint orbit  $O(\xi)$ has a natural structure of a Lie groupoid.

In the later, we try to equip the groupoid structure on the co-adjoint orbit $O(\xi)$ over $M.$ 

\begin{definition}
Let $G \rightrightarrows M $ be a  Lie groupoid with structural mappings $\alpha, ~ \beta, ~ m,~1,~ \iota.$ Let the stabilizer $G_\xi=\{g : Ad^\ast_g \xi=\xi  \}$ is a normal Lie subgroupoid of $G.$
The orbit of co-adjoint action of the Lie groupoid $G$ has the groupoid structure with the following structural mappings:
\begin{enumerate}
\item source mapping: $ \alpha^{\prime}: O(\xi) \longrightarrow M ; ~~~~~~~ Ad_{g} ^{\ast} \xi \longmapsto \alpha (g), $
\item target mapping: $ \beta^{\prime}: O(\xi) \longrightarrow M ; ~~~~~~~ Ad_{g} ^{\ast} \xi \longmapsto \beta (g), $
\item multiplication mapping: $m^{\prime} : ( O(\xi) )_{2} \longrightarrow  O(\xi);~~~~~~~$ 
 $(Ad_{g} ^{\ast} \xi ,Ad_{h} ^{\ast} \xi ) \longmapsto Ad_{m(g,h)} ^{\ast} \xi = Ad_{gh} ^{\ast} \xi , $
\item unit mapping: $ 1^{\prime}: M \longrightarrow O(\xi) ; ~~~~~~~ p \longmapsto Ad_{1_{p}} ^{\ast} \xi , $
\item  inverse mapping: $ \iota^{\prime} : O(\xi) \longrightarrow O(\xi) ; ~~~~~~~ Ad_{g} ^{\ast} \xi \longmapsto Ad_{g^{-1}} ^{\ast} \xi.$
\end{enumerate}
We call $\alpha^{\prime} , \beta^{\prime} , m^{\prime}, 1^{\prime}$ and $\iota^{\prime},$ source, target, multiplication, unit and inverse mapping, respectively, for groupoid $O(\xi).$

\end{definition}

\begin{theorem}
The orbits of co-adjoint action of any regular Lie groupoid, are Lie groupoids.
\end{theorem}

\begin{proof}
see \cite{H A 2}.
\end{proof}

We will call this Lie groupoid \textbf{co-adjoint Lie groupoid} and for simplicity we denote it by $\mathcal{G}  \rightrightarrows M $.\\
 
 \begin{lemma} 
Let $G \rightrightarrows M$ be a Lie groupoid and $m$ be its multiplication mapping. Suppose that $X \in T_{g} G,$ $ Y \in T_{h} G$ where $T \alpha (X) = T \beta (Y).$ For arbitrary members $\eta_{1} = ad_{X} ^{\ast} \xi,~ \eta_{2} = ad_{Y} ^{\ast} \xi $ such that $T \alpha^{\prime} ( \eta_{1}  ) = T \beta^{\prime} ( \eta_{2} ), $ we have that
\begin{eqnarray}
\label{Tm}
T m^{\prime} ( \eta_{1} , \eta_{2} ) = ad_{Tm ( X,Y)} ^{\ast} \xi,
\end{eqnarray}
where $m^{\prime}$ is the multiplication mapping of the co-adjoint Lie gruopoid $ \mathcal{G}  \rightrightarrows M.$
\end{lemma}
\begin{proof}
See \cite{H A 2}.
\end{proof}

\begin{definition}
Let $\mathcal{G} := O(\xi) \rightrightarrows M $ be co-adjoint Lie groupoid. According to definition \ref{associated Lie algebroid}, we define its associated Lie algebroid as follows:
$$ A \mathcal{G} = ker T\alpha^{\prime} \vert_{Ad_{1_{p}} ^{\ast} \xi} = T_{\xi} O(\xi) \vert_{Ad_{1_{p}} ^{\ast} \xi} $$
 where $T_{\xi} O(\xi) = \lbrace ad_{X} ^{\ast} \xi ~ \vert ~ X \text{	~runs over the tangent space} \; T_{g} G \rbrace.$\\
  We call this Lie algebroid \textbf{co-adjoint Lie algebroid} associated with the co-adjoint Lie groupoid. 
   \end{definition}
  
 \begin{definition}
Let $g^{\prime} = Ad_{g} ^{\ast} \xi, ~ h^{\prime} = Ad_{h} ^{\ast} \xi $ be two elements of $\mathcal{G}$ where $h , g \in G.$ The left translation corresponding to $g^{\prime}$ is 
\begin{eqnarray}
L_{g^{\prime}} ^{\prime} : \mathcal{G}^{p} & \longrightarrow & \mathcal{G}^{q}  \nonumber \\
h^{\prime} & \longmapsto & L_{g^{\prime}} ^{\prime} (h^{\prime}) = Ad_{gh} ^{\ast} \xi  \nonumber
\end{eqnarray}
and the right translation corresponding to $g^{\prime}$ is  
\begin{eqnarray}
R_{g^{\prime}} ^{\prime} : \mathcal{G}_{q} & \longrightarrow & \mathcal{G}_{p}  \nonumber \\
h^{\prime} & \longmapsto & R_{g^{\prime}} ^{\prime} (h^{\prime}) = Ad_{hg} ^{\ast} \xi.   \nonumber
\end{eqnarray}
where $p, q \in M$ and $\mathcal{G}^{p}, \mathcal{G}_{p}$ are $\beta$-fiber and $\alpha$-fiber, respectively.
\end{definition}

\begin{definition}
\label{section}
Let $ G \rightrightarrows M $ be a Lie groupoid, $AG$ be its associated Lie algebroid and $\overrightarrow{X}~( resp.,\overleftarrow{X})$ be right invariant (resp., left invariant) vector field on $G$ corresponding to $ X \in \Gamma (AG).$ Consider the co-adjoint Lie groupoid $\mathcal{G}$ and its associated Lie algebroid $A \mathcal{G}$.  We define section $X^{\prime}$ of vector bundle $\tau\prime : A \mathcal{G} \longrightarrow M$ as follows:
\begin{eqnarray}
X^{\prime} : M & \longrightarrow & A \mathcal{G} \nonumber \\
x & \longmapsto & ad_{X(x)} ^{\ast} \xi.  \nonumber
\end{eqnarray}
\end{definition}

As described in \cite{H A 2}, the right invariant vector field corresponding to $X^{\prime}$ on $\mathcal{G}$ is 
\begin{eqnarray}
\overrightarrow{X^{\prime}}(g^{\prime}) &=& T R_{g^{\prime}} ^{\prime} ( X^{\prime} (\beta^{\prime} (g^{\prime}))) \nonumber  \\
&=& ad_{\overrightarrow{X} (g)} ^{\ast} \xi  \nonumber
\end{eqnarray}
where $g \in G, ~~ g^{\prime} = Ad_{g} ^{\ast} \xi \in O(\xi),$  and
\begin{eqnarray}
\overleftarrow{X^{\prime}}(g^{\prime}) &=& - T (L_{g^{\prime}} ^{\prime})  T (\iota ^{\prime}) ( X^{\prime} ( \alpha^{\prime} (g^{\prime}))  \nonumber  \\
&=& ad_{\overleftarrow{X} (g)} ^{\ast} \xi.  \nonumber
\end{eqnarray}

Now, similar to definition \ref{associated Lie algebroid}, we define Lie algebroid structure $ ([\vert ~, ~\vert]^{\prime}, \rho^{\prime} )$ on $A \mathcal{G}$ as follows:
\begin{enumerate}
\item The anchor map: For $X \in \Gamma (\tau)$ and $ x \in M,$ 
$$
\rho^{\prime}: A \mathcal{G}  \longrightarrow  TM $$
$$ \rho^{\prime} (X^{\prime})(x) = T_{1^{\prime} (x)} \beta^{\prime} (X^{\prime} (x)) 
= \rho (X(x))  $$

\item Lie bracket: Let $ X^{\prime}= ad_{X} ^{\ast} \xi $ and $ Y^{\prime}= ad_{Y} ^{\ast} \xi $, so we have that 
$$\overrightarrow{[\vert ad_{X} ^{\ast} \xi , ad_{Y} ^{\ast} \xi} \vert ]^{\prime} = [ \overrightarrow{ ad_{X} ^{\ast} \xi } , \overrightarrow{ ad_{Y} ^{\ast} \xi } ]~~~~~( [\vert \overleftarrow{X^{\prime},Y^{\prime}} \vert]^{\prime} = - [ \overleftarrow{X^{\prime}} , \overleftarrow{Y^{\prime}} ] )$$
\end{enumerate}
Note that the right-hand bracket is the bracket on the vector fields.
 \begin{lemma}
 Consider the Lie algebroids $( AG, [\vert ~,~ \vert] , \rho )$ and $( A\mathcal{G}, [\vert ~,~ \vert]^{\prime} , \rho^{\prime} )$ and let $ X , Y \in \Gamma (AG)$ and $ X^{\prime} , Y^{\prime} \in \Gamma (A\mathcal{G})$,  then
 $$[\vert \overrightarrow{X^{\prime} , Y^{\prime}} \vert]^{\prime} = ad_{\overrightarrow{ [\vert X , Y \vert]}} ^{\ast} \xi . $$
\end{lemma}

\begin{proof}
See \cite{H A 2}. 
\end{proof}

\section{Basis of sections and structural functions}

In this section, we want to analyze the relationship between the structure functions of the Lie algebroid $AG$ of the Lie groupoid $G \rightrightarrows M$ and the co-adjoint Lie algebroid $A \mathcal{G},$ i.e the Lie algebroid of co-adjoint Lie groupoid $\mathcal{G} \rightrightarrows M$ associated to Lie groupoid $G \rightrightarrows M$ (see \cite{H A 2}). For this purpose, after expressing concepts about the basis of a sections of vector bundles, we compare structure functions between these Lie algebroids.\\

Let $\tau : A \longrightarrow M$ be a vector bundle of rank $k,$ i.e. for all $p \in M,$ every fiber $ A_{p} = \tau^{-1} (p)$ which is a finite-dimensional vector space, has a constant dimensional $k.$ Suppose that $ U \subset M$ be an open set and $\Gamma (U)$ be the space of sections of the restriction of the bundle A to U, i.e., of the vector bundle $\tau^{-1} (U) \longrightarrow U.$

\begin{definition}
\label{basis of section}

A basis of sections for A is a set of sections $e_{1}, ..., e_{k} \in \Gamma (A)$ such that the vectors $e_{1} (p), ..., e_{k} (p) \in A_{p}$ is a basis in $A_{p}$ for each $p \in M.$
\end{definition}

As it mentioned in \cite{voronov}, if $e_{\alpha}$ be a basis of sections, then each section $X \in \Gamma (A)$ is uniquely written as $X = e_{\alpha} X^{\alpha}$ where $X^{\alpha} \in C^{\infty} (M).$

\begin{remark}
The vector bundle A admits basis of sections if and only if A is trivial. Recall that the triviality of the vector bundle A means that there is a diffeomorphism $\phi: A \longrightarrow M \times \mathbb{R}^{k}$ which maps fibers to fibers and has a linear property on each fiber. \\
Every vector bundle has a local basis of sections, that's mean for every $p \in M,$ there exists an open neighborhood $U,~ x \in U$ such that every section of  $\Gamma (U)$ admits basis of sections.\\
Every basis of sections over $U_{i}$ is called local basis for A. Every (global) section $X \in \Gamma (A)$ can be considered as restricted to U and its expansion over $e_{(i)_{\alpha}}$ as follows:
$$X = e_{(i)_{\alpha}} X_{i} ^{(\alpha)}$$
where $X_{i} ^{(\alpha)} \in ^{\infty}(M)$ (see \cite{voronov} for more details).
\end{remark}

Now, consider the  Lie algebroid $ ( \tau : AG \longrightarrow M, \rho, [\vert~,~\vert] ) $ associated to Lie groupoid $G \rightrightarrows M$ with non-zero bracket. Let $X \in \Gamma (AG)$ be a section of  the Lie algebroid $AG$ and $\lbrace e_{\alpha} \rbrace,$ where $ \alpha = 1,..., dim AG,$ be local basis of sections of $AG.$ So, according to definition (\ref{basis of section}), 
\begin{eqnarray}
\label{basis}
X = e_{\alpha} X^{\alpha}
\end{eqnarray} 
 where $X^{\alpha} \in C^{\infty} (U)$ and $U$ is an open set in $M.$
 
Furthermore, suppose that $(x^{i})$ be local coordinates on $M$ and $\lbrace e_{\alpha} \rbrace$ be a local basis of sections for $AG.$ Consider local functions $\rho_{\alpha} ^{i}, C_{\alpha \beta} ^{\gamma}$ on $M$ which are called structure functions of Lie algebroid and given by
$$\rho (e_{\alpha} ) = \rho_{\alpha} ^{i} \dfrac{\partial}{\partial x^{i}}$$
and
$$[\vert e_{\alpha}, e_{\beta} \vert] = C_{\alpha\beta} ^{\gamma} e_{\gamma}.$$

Now, consider co-adjoint Lie algebroid $( \tau^{\prime}: A\mathcal{G} \longrightarrow M, \rho^{\prime}, [\vert ~,~ \vert]^{\prime} )$  of the co-adjoint Lie groupoid $\mathcal{G} := O(\xi) \rightrightarrows M.$ Let $X^{\prime} \in \Gamma (A \mathcal{G})$ be a section of the Lie algebroid $A \mathcal{G}.$  As explained in the definition of sections of  co-adjoint Lie algebroid in \cite{H A 2}, we have that $X^{\prime} = ad^{\ast} _{X} \xi$ which $X \in \Gamma (AG).$ Due to the well-known property of the section we have that $\tau^{\prime} \circ X^{\prime} = id_{M} : M \longrightarrow M.$ For every $p \in M,$ we define $( \tau^{\prime} \circ ad^{\ast} _{X}) (p) =( \tau \circ X ) (p), $ so
$$ ( \tau^{\prime} \circ X^{\prime}) (p)=\tau^{\prime} \big( X^{\prime} (p) \big) =\tau^{\prime} \big( ad^{\ast} _{X (p)} \big) = ( \tau \circ X ) (p) = p $$ Therefore, by substituting equation (\ref{basis}) and using the linear property of $ad^{\ast}$, we have
$$X^{\prime} = ad^{\ast} _{X} \xi = ad^{\ast} _{e_{\alpha} X^{\alpha}} \xi  = X^{\alpha} ad^{\ast} _{e_{\alpha}} \xi. $$

\begin{definition}
\label{basis af section for co-adjoint}
Let $\lbrace e_{\alpha} \rbrace$ be basis of section for $AG.$ We define the basis of sections for co-adjoint Lie algebroid as follows:
$$e^{\prime} _{\alpha} = ad^{\ast} _{e_{\alpha}} \xi$$
Therefore, similar to what was stated above, every section $X^{\prime} \in \Gamma (A \mathcal{G} )$ is written as follows:
$$X^{\prime}  = X^{\alpha} e^{\prime} _{\alpha}  = X^{\alpha} ad^{\ast} _{e_{\alpha}} \xi$$
where $X^{\alpha} \in C^{\infty} ( U \subset M).$
\end{definition}

\begin{lemma}
Structure functions of the Lie algebroids $AG$ and $A \mathcal{G}$ are equal, i.e. if  $\rho_{\alpha} ^{i}, C_{\alpha \beta} ^{\gamma}$ be structure functions of the Lie algebroid $AG$ and $\rho_{\alpha} ^{\prime^{i}}, C_{\alpha \beta} ^{\prime^{\gamma}}$ be structure functions of Lie algebroid $A \mathcal{G},$ then 
$$\rho_{\alpha} ^{i} = \rho_{\alpha} ^{\prime^{i}}~~~~~~~~~and ~~~~~~~~~ C_{\alpha \beta} ^{\gamma} = C_{\alpha \beta} ^{\prime^{\gamma}} . $$
\end{lemma}
\begin{proof}
Let $(x^{i})$ be local coordinates on $M$ and $ \lbrace e_{\alpha} \rbrace $ be local basis of sections for $AG$ and $\lbrace e^{\prime} _{\alpha} \rbrace$ be local basis of sections for $A\mathcal{G}.$ So, $\rho (e_{\alpha} ) = \rho_{\alpha} ^{i} \dfrac{\partial}{\partial x^{i}}$ and $\rho^{\prime} (e^{\prime} _{\alpha} ) = \rho_{\alpha} ^{\prime^{i}} \dfrac{\partial}{\partial x^{i}}. $ Furthermore, for every $X \in \Gamma (AG)$ and $X^{\prime} \in \Gamma (A \mathcal{G}_{\xi} ),$ we have $\rho (X) = \rho^{\prime} (X^{\prime}).$ So
\begin{eqnarray}
\rho^{\prime} (e^{\prime} _{\alpha} ) &=& \rho^{\prime} (ad^{\ast} _{e_{\alpha}} \xi ) \nonumber  \\
&=& \rho (e_{\alpha} )  \nonumber \\
&=& \rho_{\alpha} ^{i} \dfrac{\partial}{\partial x^{i}}. \nonumber
\end{eqnarray}
Therefore $ \rho_{\alpha} ^{\prime^{i}} = \rho_{\alpha} ^{i}. $\\

On the other hand, let $\lbrace e_{\alpha} \rbrace$ be a local basis of sections for the Lie algebroid $AG,$ so $ [\vert e_{\alpha}, e_{\beta} \vert] = C_{\alpha\beta} ^{\gamma} e_{\gamma}.$ Also, suppose that $\lbrace e^{\prime} _{\alpha}\rbrace$ be a local basis of sections of Lie algebroid $A \mathcal{G},$ so $[\vert e^{\prime} _{\alpha}, e^{\prime} _{\beta} \vert]^{\prime} = C_{\alpha\beta} ^{\prime^{\gamma}} e^{\prime} _{\gamma}.$\\
As it mentioned in \cite{H A 2}, for Lie algebroid $( AG, [\vert ~,~ \vert] , \rho )$ and co-adjoint Lie algebroid $( A\mathcal{G}, [\vert ~,~ \vert]^{\prime} , \rho^{\prime} ),$ there is following relation between Lie brackets of sections
 $$[\vert \overrightarrow{X^{\prime} , Y^{\prime}} \vert]^{\prime} = ad_{\overrightarrow{ [\vert X , Y \vert]}} ^{\ast} \xi $$
 where $ X , Y \in \Gamma (\tau)$ and $ X^{\prime} , Y^{\prime} \in \Gamma (\tau ^{\prime})$ and recall that there is a bijection between the space of sections $\Gamma (\tau)$ and the set of right-invariant vector fields on $G$ and denote by $\overrightarrow{X}$ the right invariant vector field corresponding to section $X.$ So
 \begin{eqnarray}
 [\vert e_{\alpha} ^{\prime} , e_{\beta} ^{\prime} \vert] &=& ad^{\ast} _{[\vert e_{\alpha} , e_{\beta} \vert]} \xi   \nonumber  \\
 &=& ad^{\ast} _{C_{\alpha \beta} ^{\gamma} e_{\gamma}}  \xi   \nonumber \\
 &=& C_{\alpha \beta} ^{\gamma} ad^{\ast} _{e_{\gamma}} \xi  \nonumber  \\
 &=& C_{\alpha \beta} ^{\gamma} e^{\prime} _{\gamma}.  \nonumber
 \end{eqnarray}
 The above calculations show that
 $$ C_{\alpha \beta} ^{\prime^{\gamma}}= C_{\alpha \beta} ^{\gamma}$$
\end{proof}

\textbf{Examples:}\\
\textbf{1.} Let $\mathfrak{G}$ be a Lie group and $M$ be a manifold. Consider trivial Lie groupoid  $G := M \times \mathfrak{G} \times M \rightrightarrows M.$ As described in \cite{Mackenzie}, the Lie algebroid associated to trivial Lie algebroid is $A G = TM \oplus ( M \times \mathfrak{g} ).$ The anchor $\rho :  TM \oplus ( M \times \mathfrak{g} ) \longrightarrow TM$ is the projection $X \oplus V \longmapsto X$ and the Lie bracket on the sections of $ TM \oplus ( M \times \mathfrak{g} )$ is given by
$[\vert X \oplus V, Y \oplus W \vert] = [X,Y] \oplus \lbrace X(W) - Y (V) + [V,W]\rbrace.$\\

The co-adjoint Lie groupoid associated to the trivial groupoid is $\mathcal{G}:= M \times O(\xi^{\prime}) \rightrightarrows M,$ where $O(\xi^{\prime}) $ is the orbit of co-adjoint action of Lie group. The Lie algebroid of this co-adjoint Lie groupoid is $A \mathcal{G}:= M \times T_{\xi^{\prime}} O(\xi^{\prime}).$ The anchor $\rho^{\prime}$ is given by
\begin{eqnarray}
\label{anchor}
&~&\rho : M \times T_{\xi^{\prime}} O (\xi^{\prime} )  \longrightarrow  TM \nonumber  \\
&~&\rho^{\prime}  (x , ad_{V} ^{\ast}  \xi^{\prime} ) (p) = X (p)
\end{eqnarray}
where $X \in \Gamma (TM)= \mathfrak{X}(M)$ is equal to $\dot{p}(0)$, $p(t)=\beta(\gamma(t))\in M$, $\gamma(t)=(p(t), Ad_{a} ^{\ast}  \xi^{\prime}, p(t) )\in \mathcal{G}$, $\frac{d}{dt}|_{t=0}(\gamma(t))=(x, ad_{V} ^{\ast}  \xi^{\prime})\in A \mathcal{G}$ and $ p(0)=p \in M.$\\ 
The Lie bracket on the space of sections of $ A \mathcal{G}_{\xi} $ is
\begin{eqnarray}
\label{bracket}
[\vert  ad_{V} ^{\ast}  \xi^{\prime} , ad_{W} ^{\ast}  \xi^{\prime}  \vert]^{\prime} = ad^{\ast} _{[V,W]} \xi^{\prime} 
\end{eqnarray}
for every $\Sigma^{\prime} _{1} = ad_{V} ^{\ast}  \xi^{\prime} ,~ \Sigma^{\prime} _{2} = ad_{W} ^{\ast}  \xi^{\prime}  \in \Gamma ( M \times T_{\xi^{\prime}} O (\xi^{\prime} ) )$ (see \cite{H A 2} for more details).\\

Let $(x^{i})$ be local coordinates on $M,~ (\dfrac{\partial}{\partial x^{i}})$ be local basis on $TM$ and $(\vartheta_{\alpha})$ be local basis on $\mathfrak{g}.$ For every section $X \oplus V \in \Gamma (A G),$ denote the local basis of these sections by $\lbrace e_{i} ^{\alpha} \rbrace$ and define as follows:
$$e_{i} ^{\alpha} = \dfrac{\partial}{\partial x^{i}} \oplus \vartheta_{\alpha} \in \Gamma (TM \oplus ( M \times \mathfrak{g} )).$$ 
Given that for standard Lie algebroid $TM \longrightarrow M,~ [ \dfrac{\partial}{\partial x^{i}} ,  \dfrac{\partial}{\partial x^{j}} ] = 0.$ Since $( \dfrac{\partial}{\partial x^{i}}) $ and $ (\vartheta_{\alpha})$ are local basis, they are independent of each other. So we will have:
\begin{eqnarray}
 [\vert e_{i} ^{\alpha}, e_{j} ^{\beta} \vert] &=&  [\vert \dfrac{\partial}{\partial x^{i}} \oplus \vartheta_{\alpha} , \dfrac{\partial}{\partial x^{j}} \oplus \vartheta_{\beta}  \vert]  \nonumber  \\
 &=& [\vartheta_{\alpha} , \vartheta_{\beta}]  \nonumber  \\
 &=& C_{(i,\alpha) (j, \beta)} ^{k,\gamma} e_{k} ^{\gamma} \nonumber  \\
 &=& \theta_{\alpha \beta} ^{\gamma} \vartheta_{\gamma} \nonumber
\end{eqnarray}
where $\theta_{\alpha \beta} ^{\gamma}$ are the structure constants of Lie algebra $\mathfrak{g}.$ In other words, $[\vert e_{i} ^{\alpha}, e_{j} ^{\beta} \vert]$ does not depend on $i,j,k$.\\

Let $e^{\prime} _{\alpha} = ad^{\ast} _{\vartheta_{\alpha}} \xi^{\prime}$ be basis of sections of $A \mathcal{G}$ and $\rho^{\prime} (e^{\prime} _{\alpha} ) = \rho^{\prime^{i}} _{\alpha} \dfrac{\partial}{\partial x^{i}}.$ According to equation (\ref{anchor}), simply conclude that $\rho^{\prime^{i}} _{\alpha} = \rho^{i} _{\alpha}.$\\

On the other hand, let  $\Sigma^{\prime} _{1} , \Sigma^{\prime} _{1} \in \Gamma (A \mathcal{G} ),$ where $\Sigma^{\prime} _{1} = ad_{V} ^{\ast}  \xi^{\prime} $ and $ \Sigma^{\prime} _{2} = ad_{W} ^{\ast}  \xi^{\prime}.$ So, according to equation (\ref{bracket}), we have
\begin{eqnarray}
[\vert  e^{\prime} _{\alpha}, e^{\prime} _{\beta} \vert]^{\prime} &=& [\vert  ad^{\ast} _{\vartheta_{\alpha}} \xi^{\prime}, ad^{\ast} _{\vartheta_{\beta}} \xi^{\prime} \vert]^{\prime}  \nonumber  \\
&=& ad^{\ast} _{[ \vartheta_{\alpha} , \vartheta_{\beta} ]} \xi^{\prime} \nonumber   \\
&=& ad^{\ast} _{\theta_{\alpha \beta} ^{\gamma} \vartheta_{\gamma}} \xi^{\prime}   \nonumber   \\
&=& \theta_{\alpha \beta} ^{\gamma} ad^{\ast} _{\vartheta_{\gamma}} \xi^{\prime}   \nonumber  \\
&=& \theta_{\alpha \beta} ^{\gamma} e^{\prime} _{\gamma}. \nonumber
\end{eqnarray}

\textbf{2.} Let $\mathfrak{G}$ be a Lie group, $\pi: P \longrightarrow M$ be a $\mathfrak{G}$-principal bundle and $\mathfrak{G}$ acts on $P \times P$ to the left. Consider the gauge groupoid $G:= \dfrac{P \times P}{\mathfrak{G}} $ over $M$ with its structural mappings, where $\dfrac{P \times P}{\mathfrak{G}}$ is the orbit space of the diagonal action of $\mathfrak{G}$ on $P \times P$ (see \cite{Mackenzie} for more details). The the Lie algebroid associated to gauge Lie algebroid is $A G = \dfrac{TP}{\mathfrak{G}}.$ The anchor $\rho :  \dfrac{TP}{\mathfrak{G}} \longrightarrow TM$ is given by the quotient of the natural projection map $T \pi: TP \longrightarrow TM,$ That is, $\rho = T \pi^{\sim}: \dfrac{T P}{\mathfrak{G}} \longrightarrow TM,$ and the space of sections of the vector bundle $ \tau: \dfrac{TP}{\mathfrak{G}} \longrightarrow \dfrac{P}{\mathfrak{G}} =M $ is given by
$\Gamma( \tau) = \lbrace X \Big\vert ~ X ~is~ \mathfrak{G}-invariant~ vector~ field~ on~ P \rbrace. $
Under this identification, the Lie bracket on the space of sections of  $\dfrac{TP}{\mathfrak{G}}$is given by the standard Lie bracket of vector fields.\\
Now, consider the co-adjoint Lie groupoid associated to the gauge groupoid. Recall that all co-adjoint orbits of the gauge groupoid $\dfrac{P \times P}{\mathfrak{G}}$ are isomorphic to $\dfrac{P \times O(\xi^{\prime})}{\mathfrak{G}},$ for a co-adjoint orbit $O(\xi^{\prime}) \subset \mathfrak{g}^{\ast}$ (see \cite{Bos}). Furthermore, $P / \mathfrak{G}$ may be identified with the $M,$ in this case, we have $ \mathcal{G}= M \times O (\xi^{\prime} ).$ So, its associated co-adjoint Lie algebroid is $A \mathcal{G} =\dfrac{P \times T_{\xi^{\prime}} O(\xi^{\prime})}{\mathfrak{G}} = M \times T_{\xi^{\prime}} O (\xi^{\prime} ) $ and we will reach the same results as the previous example.\\

Assuming that $\lbrace e_{i} \rbrace $ be local basis of sections for $AG,$ so $[ e_{i} , e_{j}] = C_{ij} ^{k} e_{k}.~$
 Let $(p_{i})$ be local coordinates on $P$ and $( \dfrac{\partial}{\partial p_{i}} )$ be $G$-invariant local basis on $T P.$ In this case, considering $e_{i} := \dfrac{\partial}{\partial p_{i}}$, we will have
 $$[ e_{i}, e_{j} ] = [ \dfrac{\partial}{\partial p_{i}} ,  \dfrac{\partial}{\partial p_{j}} ] = 0. $$ 
 Therefore, it follows that $C_{ij} ^{k} = 0.$\\
 On the other hand, for the vector bundle $\tau^{\prime} : \dfrac{P \times T_{\xi^{\prime}} O(\xi^{\prime})}{\mathfrak{G}} = M \times T_{\xi^{\prime}} O (\xi^{\prime} ) \longrightarrow M = \dfrac{P}{\mathfrak{G}}, $ if $\lbrace e^{\prime} _{i} \rbrace $ be local basis of sections $X^{\prime} \in \Gamma ( A \mathcal{G}_{\xi}),$ then
\begin{eqnarray}
[ e^{\prime} _{i}, e^{\prime} _{j} ] &=& ad^{\ast} _{[ e_{i}, e_{j} ]} \xi  \nonumber   \\
&=& ad^{\ast} _{C_{ij} ^{k} e_{k}} \xi   \nonumber  \\
&=& C_{ij} ^{k} ~ad^{\ast} _{e_{k}} \xi \nonumber  \\
&=& 0. \nonumber
\end{eqnarray}
Therefore, according to $[ e^{\prime} _{i} , e^{\prime} _{j}] = C_{ij} ^{\prime^{k}} e^{\prime} _{k}$, it follows that $C_{ij} ^{\prime^{k}} = 0 = C_{ij} ^{k},$
and for the anchor map $\rho^{\prime}: A \mathcal{G} \longrightarrow TM,$ from the equality  $\rho^{\prime} (X^{\prime} ) = \rho (X)$, it simply follows that $\rho^{\prime^{i}} _{\alpha} = \rho^{i} _{\alpha}.$\\

\textbf{3.} Consider action groupoid $\mathfrak{G} \ltimes M \rightrightarrows M,$ for smooth action of Lie group $\mathfrak{G}$ on a manifold $M.$ This groupoid is regular if and only if the action is transitive. If so, $M$ may be identified with the homogeneous space $\mathfrak{G} / \mathcal{H},$ where $\mathcal{H}$ is closed subgroup of $\mathfrak{G}.$ Under this identification, $\mathfrak{G} \ltimes (\mathfrak{G} / \mathcal{H} )$ is isomorphic to the gauge groupoid $\mathfrak{G} ( \mathfrak{G}  / \mathcal{H} , \mathcal{H} )$ (see \cite{Mackenzie} for more details) with the well-known results given in the previous two examples.

\section{Integrable Hamiltonian system on co-adjoint Lie algebroid}

In this section, we will first have a brief overview of the structure of a linear Poisson on dual of co-adjoint Lie algebroid. Then, using the above, we try to define a Hamiltonian function on dual of co-adjoint Lie algebroid associated to the co-adjoint Lie groupoid of the form $O(\xi) = M \times O(\xi^{\prime})$, where $O(\xi^{\prime}) \subset \mathfrak{g}^{\ast}$ is co-adjoint orbit of Lie group $\mathfrak{G}.$ In the last part of this section, we try to achieve an integral Hamiltonian system.

\subsection{Linear Poisson structure and Hamiltonian function}
Let $( \tau:A \longrightarrow M, \rho, [\vert ~,~ \vert] )$ be a Lie algebroid and $\tau^{\ast} : A^{\ast} \longrightarrow M$ be its dual bundle. It is well-known that $A^{\ast}$ admits a linear Poisson structure which is indicated by $\lbrace. , . \rbrace _{A^{\ast}},$ is characterized by the following conditions:
$$ \lbrace. , . \rbrace _{A^{\ast}} : C^{\infty} (A^{\ast} )\times C^{\infty} (A^{\ast}) \longrightarrow C^{\infty} (A^{\ast} )$$
$$ \lbrace \hat{X} , \hat{Y} \rbrace _{A^{\ast}} = - [\vert \widehat{ X , Y} \vert], $$ 
$$ \lbrace l \circ \tau^{\ast}  ~,~ \hat{X} \rbrace _{A^{\ast}} = ( \rho (X) (l) ) \circ \tau^{\ast}, $$
$$\lbrace l \circ \tau^{\ast} ~,~  k \circ \tau^{\ast}  \rbrace _{A^{\ast}} = 0.$$
where $X, Y \in \Gamma (A),~ l , k \in C^{\infty} (M)$ and $\hat{X}, \hat{Y} \in C^{\infty} (A^{\ast}).$

Also linear Poisson bivector on $A^{\ast}$ is defined by
$$\Pi_{A^{\ast}} (d \varphi , d \psi ) = \lbrace \varphi , \psi \rbrace _{A^{\ast}} $$
where $\varphi , \psi \in C^{\infty} (A^{\ast}).$\\

Let $(x^{i})$ is local coordinates on open subset $U$ of $M,~ \lbrace e_{\alpha} \rbrace$ is local basis of sections for $A,$ we have that 

\begin{equation}
\label{Poisson}
\Pi_{A^{\ast}} = \rho_{\alpha} ^{i}  \dfrac{\partial}{\partial x^{i}} \wedge \dfrac{\partial}{\partial y_{\alpha}} - \dfrac{1}{2} C_{\alpha \beta} ^{\gamma} y_{\gamma} \dfrac{\partial}{\partial y_{\alpha}} \wedge \dfrac{\partial}{\partial y_{\beta}} 
\end{equation}
where $(x^{i} , y_{\alpha})$ are the corresponding local coordinates on $A^{\ast}$ and $\rho_{\alpha} ^{i}, C_{\alpha \beta} ^{\gamma}$ are the local structure functions of $A$ with respect to the coordinates $(x^{i} )$ and basis $  \lbrace e_{\alpha} \rbrace .$\\

Moreover, in local coordinates $(x^{i} , y_{\alpha})$ on $A^{\ast}$ we have that
\begin{equation*}
\lbrace x^{i} , x^{j}  \rbrace _{A^{\ast}} = 0,~~~~~~ \lbrace  y_{\alpha} , x^{j} \rbrace _{A^{\ast}} = \rho_{\alpha} ^{j},~~~~ and ~~~~ \lbrace y_{\alpha} , y_{\beta} \rbrace _{A^{\ast}} = C_{\alpha \beta} ^{\gamma} y_{\gamma}.
\end{equation*}

For Hamiltonian function $H: A^{\ast} \longrightarrow \mathbb{R},$  the Hamiltonian vector field associated to $\Pi_{A^{\ast}}$ is as follows:
\begin{equation}
\label{HVF}
 \mathcal{X}_{H} ^{\Pi_{A^{\ast} }} (F) = \lbrace F, H \rbrace _{A^{\ast} } = \Pi_{A^{\ast} } ( d F , d H )
\end{equation}
where $ F \in C^{\infty} (A^{\ast}).$
From equations (\ref{Poisson}) and (\ref{HVF}) it follows that the local expression of $ \mathcal{X}_{H} ^{\Pi_{A^{\ast} }}$ is: 
\begin{equation}
\label{LHVF}
 \mathcal{X}_{H} ^{\Pi_{A^{\ast} }} = \dfrac{\partial H}{\partial y_{\alpha}}  \rho_{\alpha} ^{i}  \dfrac{\partial}{\partial x^{i}}  - \Big( \dfrac{\partial H}{\partial x^{i}} \rho_{\alpha} ^{i} +  \dfrac{\partial H}{\partial y_{\beta}} C_{\alpha \beta} ^{\gamma} y_{\gamma} \Big) \dfrac{\partial}{\partial y_{\alpha}}.
\end{equation}
So, the Hamiltonian equations are 
\begin{equation*}
 \dfrac{d x^{i}}{d t}= \dfrac{\partial H}{\partial y_{\alpha}}  \rho_{\alpha} ^{i}, ~~~~~~~~~\dfrac{d y_{\alpha}}{d t} = - \Big( \dfrac{\partial H}{\partial x^{i}} \rho_{\alpha} ^{i} +  \dfrac{\partial H}{\partial y_{\beta}} C_{\alpha \beta} ^{\gamma} y_{\gamma} \Big).
\end{equation*}

Let $G \rightrightarrows M$ be a Lie groupoid and $\mathcal{G}:= O(\xi) \rightrightarrows M$ be its associated co-adjoint Lie groupoid. Moreover, Let $\mathfrak{G}$ be a Lie group, $\mathfrak{g}$ its Lie algebra and $O(\xi^{\prime}) \subset \mathfrak{g}^{\ast}$ is co-adjoint orbit. 
We discussed in full detail in \cite{H A 2} that for co-adjoint Lie algebroid $ (A \mathcal{G}, \rho^{\prime}, [\vert ~,~ \vert]^{\prime} ),$ its dual bundle, $A^{\ast} \mathcal{G}$ has linear Poisson structure.

For a section $X^{\prime}=ad_{X} ^{\ast} \xi$  of  $\Gamma(\tau')$ we consider the associated linear function $\hat{X^{\prime}}$ on $A^{\ast} \mathcal{G}$ as follows:
$$\hat{X^{\prime}} : A^{\ast} \mathcal{G} \longrightarrow \mathbb{R} $$
 $$\hat{X^{\prime}} (\delta) = \delta ( X^{\prime} ( \tau^{\prime^{\ast}} (\delta))) $$
 where $\delta \in A^{\ast} \mathcal{G}$ and $\tau^{\prime^{\ast}} : A^{\ast} \mathcal{G} \longrightarrow M$ is dual bundle of $\tau^{\prime} : A \mathcal{G} \longrightarrow M.$
 In other words, the above formula indicates that $\hat{X^{\prime}} = \widehat{ad^{\ast} _{X} \xi}.$
 
 Also, $A^{\ast} \mathcal{G}$ has linear Poisson structure as follows:
 $$ \lbrace. , . \rbrace _{A^{\ast} \mathcal{G}} : C^{\infty} (A^{\ast} \mathcal{G})\times C^{\infty} (A^{\ast} \mathcal{G} ) \longrightarrow C^{\infty} (A^{\ast} \mathcal{G}) $$
\begin{equation}
\label{first eq}
\lbrace \hat{X^{\prime}} , \hat{Y^{\prime}} \rbrace _{A^{\ast} \mathcal{G}}  = - [\vert \widehat{ X^{\prime} , Y^{\prime}} \vert]^{\prime} = - \widehat{ad_{[\vert X , Y \vert]} ^{\ast} \xi} 
\end{equation}
$$ \lbrace f \circ \tau^{\prime^{\ast}}  ~,~ \hat{X^{\prime}} \rbrace _{A^{\ast} \mathcal{G}} = ( \rho^{\prime} (X^{\prime}) (f) ) \circ \tau^{\prime^{\ast}} $$
$$\lbrace f \circ \tau^{\prime^{\ast}} ~,~  g \circ \tau^{\prime^{\ast}}  \rbrace _{A^{\ast} \mathcal{G}} = 0.$$
where
 $\tau^{\prime^{\ast}} : A^{\ast} \mathcal{G} \longrightarrow M,$
 $ f , g \in C^{\infty} (M)$ 
 and
  $f \circ \tau^{\prime^{\ast}} , g \circ \tau^{\prime^{\ast}} \in C^{\infty} (A^{\ast} \mathcal{G}).$\\
  For every Hamiltonian $H^{\prime} : A^{\ast} \mathcal{G} \longrightarrow \mathbb{R}$ and $ F^{\prime} \in C^{\infty} (A^{\ast} \mathcal{G}),$ the Hamiltonian vector field $\mathcal{X}_{H^{\prime}} $ on $A^{\ast}\mathcal{G} $ will be considered as 
$$\mathcal{X}_{H^{\prime}} (F^{\prime}) = \lbrace F, H^{\prime} \rbrace _{A^{\ast} \mathcal{G}} = \Pi_{A^{\ast} \mathcal{G}} ( d F^{\prime} , d H^{\prime} ).$$ 

Also there, as a specific example, we considered the trivial Lie groupoid showed that the co-adjoint Lie groupoid associated to this Lie groupoid is of the form $O(\xi) = M \times O(\xi^{\prime}) $ with co-adjoint Lie algebroid $A \mathcal{G} = M \times T_{\xi^{\prime}} O(\xi^{\prime}).$
Similar results were obtained for gauge groupoid and specific case for the action groupoid.\\

Given the first property of the linear Poisson structure, equation (\ref{first eq}), we try to define the Hamiltonian function and the first integral function. \\

So, according to what was described above, in the following lemma, we will show that if the co-adjoint Lie groupoid associated to a Lie groupoid be of the form $O(\xi) = M \times O(\xi^{\prime}), $ then the Hamiltonian function $H: M \times T_{\xi^{\prime}} ^{\ast} O(\xi^{\prime})  \longrightarrow \mathbb{R}$ is equal to Hamiltonian $h: T^{\ast} O(\xi^{\prime}) \longrightarrow \mathbb{R}. $

\begin{lemma}
\label{H h}
If $O(\xi) = M \times O(\xi^{\prime})$ be  the co-adjoint Lie groupoid associated to a Lie groupoid $G \rightrightarrows M,$ then the Hamiltonian function $H: M \times T_{\xi^{\prime}} ^{\ast} O(\xi^{\prime})  \longrightarrow \mathbb{R}$ does not depend on the coordinates of $M,$ which $O(\xi^{\prime}) \subset \mathfrak{g}^{\ast}$ is co-adjoint orbit and $\mathfrak{g}^{\ast}$ is dual of Lie algebra corresponding to Lie group $\mathfrak{G}.$
\end{lemma}
\begin{proof}
Let $(\mathfrak{y}_{\alpha})$ be local coordinates on $T^{\ast} O(\xi^{\prime}) $ and $(x^{i})$ be local coordinates on $M.$ 
So, similar to what was stated in equation (\ref{LHVF}), the Hamiltonian vector field $\mathcal{X}_{H} \in \mathfrak{X} ( M \times T^{\ast} _{\xi^{\prime}} O(\xi^{\prime} ))$ for Hamiltonian $H: M \times T_{\xi^{\prime}} ^{\ast} O(\xi^{\prime})  \longrightarrow \mathbb{R}$ is as follows:
\begin{equation*}
\mathcal{X}_{H} ^{\Pi_{A^{\ast} \mathcal{G}_{\xi}}} =  \dfrac{\partial H}{\partial x^{i}}   \dfrac{\partial }{\partial x^{i}}  -  C_{\alpha \beta} ^{\gamma} \mathfrak{y}_{\gamma}  \dfrac{\partial H}{\partial \mathfrak{y_{\beta}}}  \dfrac{\partial} {\partial  \mathfrak{y}_{\alpha}}.
\end{equation*}

 Thus, the corresponding Hamiltonian equations are as follows:
 \begin{equation}
 \label{HE1}
\dfrac{d x^{i}}{dt} = \dfrac{\partial H}{\partial x^{i}}, ~~~~~~~~~~~~~\dfrac{d \mathfrak{y}_{\alpha}}{dt} = - C_{\alpha \beta} ^{\gamma} \mathfrak{y}_{\gamma} \dfrac{\partial H}{\partial \mathfrak{y}_{\beta}}.
 \end{equation}
 Moreover, if $(x^{i}) $ are local coordinates on $M, \lbrace e^{\prime} _{\alpha} \rbrace $ be local basis of $\Gamma (A \mathcal{G} )$ and $(x^{i}, \mathfrak{y}_{\alpha})$ are corresponding coordinates on $ A^{\ast} \mathcal{G} = M \times T^{\ast} _{\xi^{\prime}} O(\xi^{\prime} )$ then the local expression of $\Pi_{A^{\ast} \mathcal{G}} $ will be as follows:
 \begin{equation*}
 \label{C P}
  \Pi_{A^{\ast} \mathcal{G}_{\xi}}  = \dfrac{1}{2} \dfrac{\partial}{\partial x^{i}} \wedge \dfrac{\partial}{\partial x^{i}} - \dfrac{1}{2} C_{\alpha \beta} ^{\gamma} \mathfrak{y}_{\gamma} \dfrac{\partial}{\partial \mathfrak{y}_{\alpha}} \wedge \dfrac{\partial}{\partial \mathfrak{y}_{\beta}}. 
 \end{equation*}
On the other hand, let $h: T^{\ast} O(\xi^{\prime}) \longrightarrow \mathbb{R}$ be Hamiltonian function on $T^{\ast} O(\xi^{\prime}).$ In local coordinates $ \mathfrak{y}_{\alpha} $ for $ T^{\ast} O(\xi^{\prime}), $ the Poisson 2-vector $ \Pi $ is
\begin{equation*}
\Pi = - \dfrac{1}{2}  C_{\alpha \beta} ^{\gamma} \mathfrak{y}_{\gamma} \dfrac{\partial}{\partial \mathfrak{y}_{\alpha}} \wedge \dfrac{\partial}{\partial \mathfrak{y}_{\beta}},
\end{equation*}
and the Hamiltonian vector field $ \mathcal{X}_{h} $ associated to Hamiltonian function $ h $ is
\begin{equation*}
\mathcal{X}_{h} = - C_{\alpha \beta} ^{\gamma} \mathfrak{y}_{\gamma}  \dfrac{\partial H}{\partial \mathfrak{y_{\beta}}}  \dfrac{\partial} {\partial  \mathfrak{y}_{\alpha}}.
\end{equation*}
 So, its Hamiltonian equations are as follows:
\begin{equation}
\label{HE2}
\dfrac{d \mathfrak{y}_{\alpha}}{dt} = - \dfrac{\partial h}{\partial \mathfrak{y}_{\beta}} C_{\alpha \beta} ^{\gamma} \mathfrak{y}_{\gamma}.
\end{equation}
Suppose $\lbrace~,~ \rbrace_{K.K} $  is the symbol of Kirillov-Kostant bracket on $C^{\infty} ( T^{\ast} _{\xi^{\prime}} O(\xi^{\prime} )).$ It is clear that
\begin{equation}
\label{hvf}
 \mathcal{X}_{h} (f ) = \Pi (d f , d h ) = \lbrace f, h \rbrace_{K.K}.
\end{equation}
 Furthermore, as we proved in \cite{H A 2},  Hamiltonian $H: M \times T_{\xi^{\prime}} ^{\ast} O(\xi^{\prime})  \longrightarrow \mathbb{R},$ 
 $$\lbrace F, H \rbrace _{A^{\ast} \mathcal{G} } = \lbrace f, h \rbrace_{K.K}$$
 where $ F = (m , f ) \in C^{\infty} (M \times T^{\ast} _{\xi^{\prime}} O(\xi^{\prime} )) $ and $ f \in C^{\infty} (T^{\ast} _{\xi^{\prime}} O(\xi^{\prime} )).$
 therefore, according to equations (\ref{HVF}) and (\ref{hvf}), we conclude that
  \begin{equation*}
 \label{equality}
  \mathcal{X}_{H} ^{\Pi_{A^{\ast} \mathcal{G}}} (F) =\mathcal{X}_{h} (f )
 \end{equation*}
where $ \mathcal{X}_{H} ^{\Pi_{A^{\ast} \mathcal{G}}}$ is Hamiltonian vector field associated to Hamiltonian $H: M \times T_{\xi^{\prime}} ^{\ast} O(\xi^{\prime})  \longrightarrow \mathbb{R}.$ 
From this equality and comparison of equations (\ref{HE1}) and (\ref{HE2}), it follows that $\dfrac{\partial H}{\partial x^{i}} = 0,$ i.e. $H: M \times T_{\xi^{\prime}} ^{\ast} O(\xi^{\prime})  \longrightarrow \mathbb{R}$ does not depend on the coordinates of $M.$
\end{proof}
\\

Therefore, according to the lemma mentioned above (lemma \ref{H h}), we consider the Hamiltonian function $H: M \times T_{\xi^{\prime}} ^{\ast} O(\xi^{\prime})  \longrightarrow \mathbb{R}$ such that $H = (p, h),$ i.e. for $\delta = (p,\lambda) \in M \times T_{\xi^{\prime}} ^{\ast} O(\xi^{\prime}),$ we have that $H (p, \lambda ) = h(\lambda),$
where $h: T^{\ast} O(\xi^{\prime}) \longrightarrow \mathbb{R}$ is Hamiltonian function on $T^{\ast} O(\xi^{\prime}).$\\
 In the following, we try to define the first integral corresponding to Hamiltonian vector field associated to the Hamiltonian $H.$

\begin{definition}
\label{first integral}
Let  $H: M \times T_{\xi^{\prime}} ^{\ast} O(\xi^{\prime})  \longrightarrow \mathbb{R}$ is the Hamiltonian function and $ \mathcal{X}_{H} ^{\Pi_{A^{\ast} \mathcal{G}}}$ is Hamiltonian vector field associated to Hamiltonian $H.$ A function $F \in C^{\infty} (M \times T^{\ast} _{\xi^{\prime}} O(\xi^{\prime} ))$ is called first integral corresponding to Hamiltonian vector field $\mathcal{X}_{H} ^{\Pi_{A^{\ast} \mathcal{G}}}$ whenever
\begin{equation*}
\mathcal{X}_{H} ^{\Pi_{A^{\ast} \mathcal{G}}} (F) = \lbrace F, H \rbrace_{A^{\ast} \mathcal{G}} =0.
\end{equation*}
\end{definition}

\begin{remark}
\label{L F}
As we stated in \cite{H A 2}, every function $F \in C^{\infty} (M \times T^{\ast} _{\xi^{\prime}} O(\xi^{\prime} ))$ is defined as  $F=(p,f)$ which $f \in C^{\infty} ( T^{\ast} O(\xi^{\prime} )).$
\end{remark}
In the following statement, we want to show the relationship between the first integral of Hamiltonian vector field associated to Hamiltonian $H: M \times T_{\xi^{\prime}} ^{\ast} O(\xi^{\prime})  \longrightarrow \mathbb{R}$ and the first integral of Hamiltonian vector field corresponding to Hamiltonian $h: T^{\ast} O(\xi^{\prime}) \longrightarrow \mathbb{R}.$ 
\begin{proposition}
\label{first int pro}
Suppose that $H : M \times T_{\xi^{\prime}} ^{\ast} O(\xi^{\prime})  \longrightarrow \mathbb{R}, ~ H (p,\lambda) = h(\lambda)$ is the Hamiltonian function on $ A^{\ast} \mathcal{G},$ $~ \Pi_{A^{\ast} \mathcal{G}} $ be the Poisson structure on $A^{\ast} \mathcal{G}$ and $ \mathcal{X}_{H} ^{\Pi_{A^{\ast} \mathcal{G}}}$ is Hamiltonian vector field associated to Hamiltonian $H.$ Also, assume that  $\mathcal{X}_{h} ^{\Pi}$ be Hamiltonian vector field corresponding Hamiltonian $h: T^{\ast} O(\xi^{\prime}) \longrightarrow \mathbb{R}.$ Then, first integrals of Hamiltonian vector field $\mathcal{X}_{H} ^{\Pi_{A^{\ast} \mathcal{G}}}$ correspond to the first integrals of Hamiltonian vector field $\mathcal{X}_{h} ^{\Pi},$ i.e. $F  = (p,f) \in C^{\infty} (M \times T^{\ast} _{\xi^{\prime}} O(\xi^{\prime} ))$ is first integral of Hamiltonian vector field $\mathcal{X}_{H} ^{\Pi_{A^{\ast} \mathcal{G}}}$ if and only  $f \in C^{\infty} ( T^{\ast} O(\xi^{\prime} ))$ is first integral of Hamiltonian vector field $\mathcal{X}_{h} ^{\Pi}.$
 
\end{proposition} 
\begin{proof}
Let $f \in C^{\infty} ( T^{\ast} O(\xi^{\prime} ))$ be first integral of Hamiltonian vector field $\mathcal{X}_{h} ^{\Pi}.$ Thus $\mathcal{X}_{h} ^{\Pi} (f) = 0.$ \\
Consider $F = (p,f) \in C^{\infty} (M \times T^{\ast} _{\xi^{\prime}} O(\xi^{\prime} )).$ So we have
\begin{eqnarray}
\mathcal{X}_{H} ^{\Pi_{A^{\ast} \mathcal{G}}} (F) &=& \lbrace F, H \rbrace_{A^{\ast} \mathcal{G}}  \nonumber  \\
&=& \lbrace f,h \rbrace_{K.K} \nonumber  \\
&=& \mathcal{X}_{h} ^{\Pi} (f)  \nonumber \\
&=& 0 \nonumber
\end{eqnarray}
So it turns out that $F = (p,f)$ is first integral of Hamiltonian vector field $\mathcal{X}_{H} ^{\Pi_{A^{\ast} \mathcal{G}_{\xi}}}.$\\

 On the other hand, consider $F \in C^{\infty} (M \times T^{\ast} _{\xi^{\prime}} O(\xi^{\prime} )).$ According to remark (\ref{L F}), we have that $F = (p,g),$ where $g \in C^{\infty} ( T^{\ast} O(\xi^{\prime} )).$  Suppose that $g$ is not the first integral of Hamiltonian vector field $\mathcal{X}_{h} ^{\Pi},$  i.e. $\lbrace g,h \rbrace_{K.K} \neq 0.$\\
 Let $F = (p,g)$ be first integral of Hamiltonian vector field $\mathcal{X}_{H} ^{\Pi_{A^{\ast} \mathcal{G}}}.$ Thus, according to definition (\ref{first integral}), we have that $\mathcal{X}_{H} ^{\Pi_{A^{\ast} \mathcal{G}}} (F) = 0,$ so $ \lbrace F, H \rbrace_{A^{\ast} \mathcal{G}} = 0.$ Since $ \lbrace F, H \rbrace_{A^{\ast} \mathcal{G}} =\lbrace g,h \rbrace_{K.K},$ it follows that $\lbrace g,h \rbrace_{K.K} = 0,$ which contradicts the assumption that $g$ is not the first integral of $\mathcal{X}_{h} ^{\Pi}.$ \\
So, $F  = (p,f) \in C^{\infty} (M \times T^{\ast} _{\xi^{\prime}} O(\xi^{\prime} ))$ is first integral of Hamiltonian vector field $\mathcal{X}_{H} ^{\Pi_{A^{\ast} \mathcal{G}}}$ if and only  $f \in C^{\infty} ( T^{\ast} O(\xi^{\prime} ))$ is first integral of Hamiltonian vector field $\mathcal{X}_{h} ^{\Pi}.$
\end{proof}
\begin{corollary}
The Hamiltonian vector field $ \mathcal{X}_{H} ^{\Pi_{A^{\ast} \mathcal{G}}}$ on $O(\xi)$ is integrable if only if the Hamiltonian vector field $\mathcal{X}_{h} ^{\Pi}$ on $O(\xi')$ is integrable.
\end{corollary}

\section{Symplectic structure on co-adjoint Lie groupoid}
In this section, we try to examine the symplectic structure on co-adjoint Lie groupoids. In fact, by defining a symplectic structure for the co-adjoint Lie groupoids, similar to what is said about the co-adjoint orbits of Lie groups, we intend to examine the symplectic groupoid of co-adjoint Lie groupoids.\\

Let $G \rightrightarrows M$ be a Lie groupoid with the structural functions $\alpha$, $\beta$, $1$, $\iota$ and $m.$
\begin{definition}
\label{sym groupoid}
The symplectic groupoid is a Lie groupoid $G \rightrightarrows M$ equipped with symplectic form $\omega$ on $G$ such that
$$m^{\ast} \omega = pr_{1} ^{\ast} \omega + pr_{2} ^{\ast} \omega$$
where for $i=1,2,$ $pr_{i} : G_{2} \longrightarrow G$ are the projections over the first and second factor.
\end{definition}

\begin{example}
\label{sym ex}
Every Lie groupoid $G \rightrightarrows M$ with the associated Lie algebroid $AG$ induces a Lie groupoid $T^{\ast} G  \rightrightarrows A^{\ast} G.$ The cotangent bundle $T^{\ast} G$ with the canonical symplectic form $\omega_{G}$ is a symplectic Lie groupoid over $A^{\ast} G.$ (See \cite{Mackenzie} for more details.)

\end{example}

Consider the co-adjoint Lie groupoid $\mathcal{G}:= O(\xi) \rightrightarrows M$ corresponding to the Lie groupoid $G \rightrightarrows M.$
Therefore, as stated in example \ref{sym ex}, it can be concluded that for the co-adjoint Lie groupoid $\mathcal{G} \rightrightarrows M$,
$T^{\ast} \mathcal{G} \rightrightarrows A^{\ast} \mathcal{G}$
 is a symplectic Lie groupoid. 

In the following, by presenting a proposition, we show that the co-adjoint Lie groupoid inherits the property of being symplectic groupoid from the symplectic groupoid $G \rightrightarrows M.$
 
\begin{proposition}
Co-adjoint Lie groupoid of a symplectic groupoid is a symplectic groupoid. 
\end{proposition}
\begin{proof}
Let $ (G,\omega) \rightrightarrows M$ be a symplectic groupoid. So, according to definition \ref{sym groupoid}, for every $X, Y \in T_{g} G$ we have,
$$ \omega \big( Tm (X,Y) \big) = \omega (X) + \omega (Y).$$
Consider the co-adjoint Lie groupoid $\mathcal{G}:= O(\xi) \rightrightarrows M$ corresponding to the Lie groupoid $G \rightrightarrows M.$ As mentioned before 
$$O(\xi) = \lbrace Ad_{g} ^{\ast} \xi ~ \vert ~ g \in G \rbrace,$$
and the tangent space to $O(\xi)$ is
$$TO(\xi) = \lbrace ad_{X} ^{\ast} \xi ~ \vert ~ X \in T_{g} G \rbrace.$$

Select two arbitrary tangent vectors 
$\eta_{1} = ad_{X} ^{\ast} \xi $ ,$\eta_{2} = ad_{Y} ^{\ast} \xi, $ and define the symplectic structure $\omega^{\prime}$ on $O(\xi)$ as follows:
$$ \omega^{\prime} ( \eta_{1} , \eta_{2} )=  \omega^{\prime} ( ad_{X} ^{\ast} \xi , ad_{Y} ^{\ast} \xi ) = \omega (X,Y).$$

Now, to prove that co-adjoint Lie groupoid $\mathcal{G}:= O(\xi) \rightrightarrows M$ is symplectic Lie groupoid, we examine relation
$$m^{\prime^{\ast}} \omega^{\prime} = pr_{1} ^{\ast} \omega^{\prime} + pr_{2} ^{\ast} \omega^{\prime}. $$

According to equation (\ref{Tm}), we have
\begin{eqnarray}
m^{\prime^{\ast}} \omega^{\prime} (\eta_{1} , \eta_{2} ) &=&  \omega^{\prime}  \big(T m^{\prime} (\eta_{1} , \eta_{2} ) \big)  \nonumber   \\
&=& \omega^{\prime} ( ad^{\ast} _{Tm (X,Y)} \xi )  \nonumber  \\
&=& \omega \big( Tm (X,Y) \big).  \nonumber
\end{eqnarray}
Also,
\begin{eqnarray}
pr_{1} ^{\ast} \omega^{\prime} (\eta_{1} , \eta_{2} ) + pr_{2} ^{\ast} \omega^{\prime} (\eta_{1} , \eta_{2} ) &=& \omega^{\prime} \big( Tpr_{1} (\eta_{1} , \eta_{2} ) \big) +  \omega^{\prime} \big( Tpr_{2} (\eta_{1} , \eta_{2} ) \big)   \nonumber  \\
&=& \omega^{\prime} ( ad_{X} ^{\ast} \xi) + \omega^{\prime} ( ad_{Y} ^{\ast} \xi)   \nonumber   \\
&=& \omega (X) + \omega (Y).   \nonumber   
\end{eqnarray}

Therefore, from the above calculations, it follows that $ O(\xi) \rightrightarrows M$ is a symplectic groupoid.
\end{proof}\\


\small{

}

\end{document}